\documentclass[aos,preprint]{imsart}
\RequirePackage[OT1]{fontenc}
\RequirePackage{amsthm,amsmath}
\RequirePackage[numbers]{natbib}
\usepackage{pictexwd}
\usepackage{bm,amssymb,graphicx,url}
\usepackage{color}
\usepackage{epstopdf}
\usepackage[normalem]{ulem}

\usepackage[colorlinks,citecolor=blue,urlcolor=blue]{hyperref}

\startlocaldefs    

\def\a{\alpha}

\def\R{\mathbb R}

\def\E{{\mathbb E}}
\def\F{{\mathbb F}}

\def\l{\lambda}
\def\labda1{\lambda_1}
\def\labda2{\lambda_2}
\def\m{\mu}

\def\e{\varepsilon}
\def\f{\phi}

\def\s{\sigma}

\def\comment#1{\relax}

\def\=in{\mathop{\rm =}}

\numberwithin{equation}{section}
\theoremstyle{plain}
\endlocaldefs

\usepackage{amsmath,here,amsfonts,latexsym,graphicx,here}
\usepackage{amsthm,curves}
\usepackage{graphicx}
\usepackage{caption}
\usepackage{subcaption}
\usepackage{float}

\def\a{\alpha}

\def\e{\epsilon}
\def\l{\lambda}
\def\m{\mu}

\def\f{\phi}

\def\s{\sigma}

\def\cS{\mathcal{S}}

\def\F{{\mathbb F}}

\def\Q{{\mathbb Q}}
\def\R{{\mathbb R}}

\def\a{\alpha}
\def\m{\mu}

\newtheorem{theorem}{Theorem}
\newtheorem{lemma}{Lemma}

\newtheorem{remark}{Remark}

\begin{document}
\begin{frontmatter}
\title{Estimation of the incubation time distribution in the singly and doubly interval censored model}
\runtitle{MLE for the incubation time distribution}

\begin{aug}
\author{\fnms{Piet} \snm{Groeneboom}\corref{}\ead[label=e1]{P.Groeneboom@tudelft.nl}}
\runauthor{P.\ Groeneboom}
\affiliation{Delft University of Technology}
\address{Delft University of Technology, Mekelweg 4, 2628 CD Delft,
	The Netherlands.\\ 
	\printead{e1}} 
\end{aug}

\begin{abstract}
We analyze nonparametric estimators for the distribution function of the incubation time in the singly and doubly interval censoring model. The classical approach is to use parametric families like Weibull, log-normal or gamma distributions in the estimation procedure. We propose nonparametric estimates for functions of the observations, which stay closer to the data than the classical parametric methods.
We also give explicit limit distributions for discrete versions of the models and apply this to compute confidence intervals.  The methods complement the analysis of the continuous model in \cite{piet:21} and \cite{piet:23}. {\tt R} scripts for computation of the estimates are provided in \cite{github:20}.
\end{abstract} 

\begin{keyword}[class=AMS]
\kwd[Primary ]{62G05}
\kwd{62N01}
\kwd[; secondary ]{62-04}
\end{keyword}

\begin{keyword}
\kwd{incubation time, confidence intervals, Fisher information, support reduction, single interval censoring, double interval censoring, deconvolution}
\end{keyword}

\end{frontmatter}

\section{Introduction}
\label{section:intro}
In the statistical analysis of the behavior of an infectious disease, one usually has to deal with events that are not directly observable. As an example, at the start of the Covid-19 pandemic, the so-called effective reproductive number $R_e$ played an important role (``Is it bigger or smaller than 1?"). The estimation of $R_e$ faces the difficulty that infection events can usually not be observed and that therefore a deconvolution step is necessary, see, e.g., \cite{maathuis:22}

For the estimation of the distribution of the incubation time one also faces the difficulty just mentioned.
 In this case one has an observation interval $[E_L,E_R]$ which is known to contain the time of infection $I$ and a time $S$ where the person becomes symptomatic. The incubation time is then given by $S-I$.

 Following \cite{tom_gianpi:19} we set the left endpoint of the exposure interval $[E_L,E_R]$ equal to zero (``looking back''). Our observations then consist of the pair of the (lengths of the) exposure time $E$ and the time of getting symptomatic $S$
$$
(E,S),
$$
where $S=I+U$,  $I$ is the infection time (also shifted for taking $E_L=0$) and $U$ the (length of the) incubation time (which does not have to be shifted). 
The times $I$ and $U$ are assumed independent, given $E$, and are not observable.

To make the distribution function $F$ of the incubation time $U$ identifiable, one has to make an assumption on the distribution function of the infection time. 
It is usually assumed  that that the time till infection is uniformly distributed on the interval $[0,E]$, conditionally on the length of the exposure time $E$. The model is for example considered in \cite{reich:09}, \cite{tom_gianpi:19}, \cite{backer:20} and \cite{piet:21}.
It is also possible to let the infection time have another distribution on $[0,E]$, but we have to make an assumption on this to make the distribution of the infection time identifiable. In the present paper we will assume that the distribution is uniform on $[0,E]$.

We define the (convolution) density $q_F$ of $(E,S)$ by
\begin{align}
\label{convolution}
q_F(e,s)&=e^{-1}\{F(s)-F(s-e)\}=e^{-1}\int_{u=(s-e)_+}^s \,dF(u),\qquad e>0,\,s\ge0.
\end{align}
w.r.t.\ $\m$, which is the product of the measure $dF_E$ of the exposure time $E$ and Lebesgue measure. The distribution function $F$ satisfies $F(x)=0$ for $x\le0$. So the underlying measure $Q_F$ for $(E,S)$ is defined by
\begin{align}
\label{def_Q_0}
dQ_F(e,s)=q_F(e,s)\,ds\,\,dF_E(e),\qquad e\in(0,M_2],\,\qquad s\ge0,
\end{align}
where $M_2<\infty$ is the upper bound of the support of $E$.

It seems reasonable to assume that the underlying distribution function $F_0$ of the incubation time is absolutely continuous, with density $f_0$, and that $E$ has no mass on an interval $[0,\e)$, for some $\e>0$. Observations of this type are shown schematically in Figure \ref{fig:picture_single}. Note that if $S<E$, one usually puts $S=E$, since then $E$ is no longer relevant for the estimation.

\begin{figure}[!ht]
\centering
\includegraphics[width=0.8\textwidth]{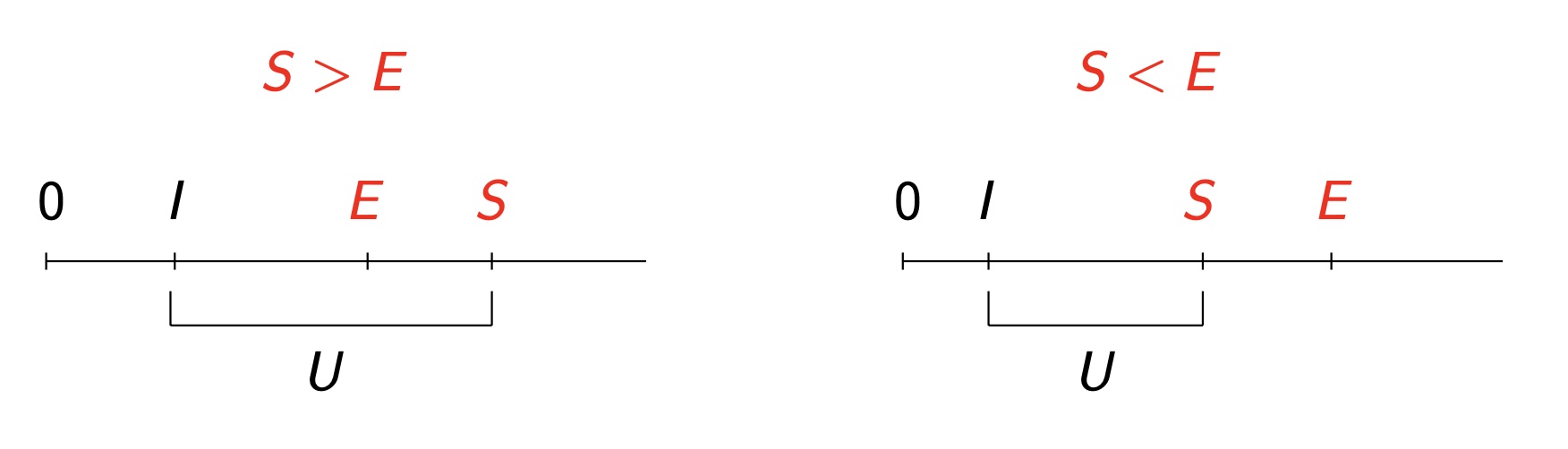}
\caption{Singly interval censored data. $E$ is the end of the exposure time, $S$ the time of becoming symptomatic, $I$ infection time and $U$ the (length of the) incubation time. We only can observe $E$ and $S$.}
\label{fig:picture_single}
\end{figure}

In practice,  the variables $S=I+U$ are usually rounded, taking the ceil of $S$, to integers (days), in which case the log likelihood for one observation becomes, conditionally on the values of $E$,
\begin{align}
\label{loglikehood2}
\log \int_{s=\lfloor S\rfloor}^{\lceil S\rceil}\,q_F(E,s)\,ds=\log \left\{E^{-1}\int_{s=\lfloor S\rfloor}^{\lceil S\rceil}\{F(s)-F(s-E)\}\,ds\right\},
\end{align}
where  $\lfloor S \rfloor$ and $\lceil S\rceil$ are the floor and ceil of $S$, respectively. Note that $S$ itself is assumed to have a continuous distribution and that $S$ therefore lies with probability one strictly between two consecutive integers, so we integrate over an interval of length 1. In this case, where we do not have more precise information on the values of the time of getting symptomatic, we call the model {\it doubly censored}: we only have an interval for the beginning of the incubation time (time of getting infected) and an interval (the 24 hours of a day) for the time of getting symptomatic. If we would have more precise information about the time of getting symptomatic, so that we can trat $S$ as a continuous (observable) random variable, we would call the model {\it singly censored}. Theory for nonparametric analysis of this model is developed in \cite{piet:23}.

Usually the exposure times $E$ are also only known as number of days, so represented by integers, as we will do in the sequel.  In principle we  can also consider a continuous exposure time $E$, but we will not do this to avoid an overcomplicated model.
So instead of the parameters $F_0(S)$ and $F_0(S-E)$, we consider estimating the parameters
\begin{align}
\label{parameters_discretized}
\bar F_0(\lceil S\rceil)=\int_{s=\lfloor S\rfloor}^{\lceil S\rceil}F_0(s)\,ds,\quad\text{ and }\quad \bar F_0(\lceil S\rceil-E)=\int_{s=\lfloor S\rfloor}^{\lceil S\rceil} F_0(s-E)\,ds,
\end{align}
where $E\in\mathbb N$, and where we use the notation $F_0$ to denote the underlying distribution function.

For a sample $(E_1,S_1), \dots,(E_n,S_n)$, where the $S_i$ are integers, and represent the ceils in (\ref{parameters_discretized}), we get the log likelihood
\begin{align}
\label{loglikehood3}
\ell(F)=\sum_{i=1}^n \log \{\bar F(S_i)-\bar F(S_i-E_i)\},
\end{align}
for distribution functions $F$, where we assume that the $E_i$ are also integers.

To give a specific example, we consider the following parametric model. The incubation time $U$ has a truncated exponential distribution function $F_a$, defined by
\begin{align*}
F_a(x)=\frac{1-\exp(-x/a)}{1-\exp(-M_1/a)}\,1_{[0,M_1]}(x) +1_{(M_1,\infty)}(x),\qquad x\in\R,
\end{align*}
for some constants $M_1,a>0$. There is no reason to use this distribution as a model for the incubation time, but neither is there a reason for using the Weibull, log-normal or gamma distributions, which are mostly used. The advantage of taking the exponential distribution in our example is that the integral in (\ref{loglikehood2}) has a simple form, in contrast with the integrals of the latter distribution functions.

Now, once we have fixed the upper bound for the incubation time $M_1$ (something that is needed in the nonparametric approach, but it also does not sound unrealistic to assume that such an upper bound exists), for example $M_1=15$, the parametric estimation of the distribution function of the incubation time boils down to the maximization of the function
\begin{align}
\label{parametric_log_likelihood}
\ell(F_a)=\sum_{i=1}^n \log \int_{s=\lfloor S_i\rfloor}^{\lceil S_i\rceil}\left\{F_a(s)-F_a(s-E_i)\right\}\,ds,
\end{align}
over $a>0$, where
\begin{align}
\label{F_a}
F_a(x)=\left\{\begin{array}{lll}
0 &, x<0,\\
\left(1-e^{-x/a}\right)/\left(1-e^{-M_1/a}\right)\qquad &, x\in[0,M_1],\\
1 &, x>M_1.\\
\end{array}
\right.
\end{align}

If $E_i<S_i$, the integrals $\int_{s=\lfloor S_i\rfloor}^{\lceil S_i\rceil}\left\{F_a(s)-F_a(s-E_i)\right\}\,ds$ are of the form
\begin{align}
\label{explicit_parametric_log_likelihood}
&\int_{s=k-1}^{k}\left\{F_a(s)-F_a(s-E_i)\right\}\,ds\nonumber\\
&=(1-a\exp(-k/a)(\exp(1/a)-1))/(1-\exp(-M_1/a)))\nonumber\\
&\qquad\qquad-(1-a\exp(-(k-j)/a)(\exp(1/a)-1))/(1-\exp(-M_1/a))
\end{align}
(note that the $E_i$ are assumed to be integers). If $E_i\ge S_i$ we get:
\begin{align}
\label{explicit_parametric_log_likelihood2}
&\int_{s=k-1}^{k}\left\{F_a(s)-F_a(s-E_i)\right\}\,ds=\int_{s=k-1}^{k}F_a(s)\,ds\nonumber\\
&=(1-a\exp(-k/a)(\exp(1/a)-1))/(1-\exp(-M_1/a))).
\end{align}

So maximization of $\ell(F_a)$ in (\ref{parametric_log_likelihood}) boils down to maximization of sums of logarithms of expressions of the form (\ref{explicit_parametric_log_likelihood}) and (\ref{explicit_parametric_log_likelihood2}) as a function of the parameter $a$. We used the {\tt R} package {\tt nloptr} for this maximization, which is an interface to the non-linear optimization package \cite{NLopt} (written in {\tt C}); the R script for our particular optimization problem can be found in \cite{github:20}.

On the other hand, the nonparametric maximum likelihood estimator maximizes
\begin{align}
\label{nonparametric_log_likelihood}
\ell_2(\bar F)&=\sum_{i=1}^n \log \int_{s=\lfloor S_i\rfloor}^{\lceil S_i\rceil}\left\{F(s)-F(s-E_i)\right\}\,ds\nonumber\\
&=\sum_{i=1}^n \log\left\{\bar F({\lceil S_i\rceil})-\bar F({\lceil S_i\rceil}-E_i)\right\}
\end{align}
over {\it all} distribution functions $F$, or, alternatively, all discrete distribution functions $\bar F$, only having jumps at the integers $i$. Since this is a maximization over a much wider class of functions, we cannot expect it to be as good as the parametric estimate for its own parametric model.

To get a feel for their relative performances, we show the box plots of the nonparametric and parametric estimates of the distribution function at the point 6 (``6 days'') for this particular parametric model. It is seen in Figure \ref{figure:boxplot500}(a) that both estimates are well on target, but that the nonparametric estimate has a bigger variance. This is to be expected, since the nonparametric estimate does not presuppose this particular parametric model, as the parametric estimate in fact does. Figure \ref{figure:boxplot500}(b) shows what happens if we take for the parametric estimate the value of $F_{\hat a}(6)$ instead of the integral $\int_5^6 F_{\hat a}(x)\,dx$, where $\hat a$ is the parametric estimate of the parameter $a$, resulting from the application of {\tt nloptr}.

\begin{figure}[!ht]
\begin{subfigure}[b]{0.45\textwidth}
\includegraphics[width=\textwidth]{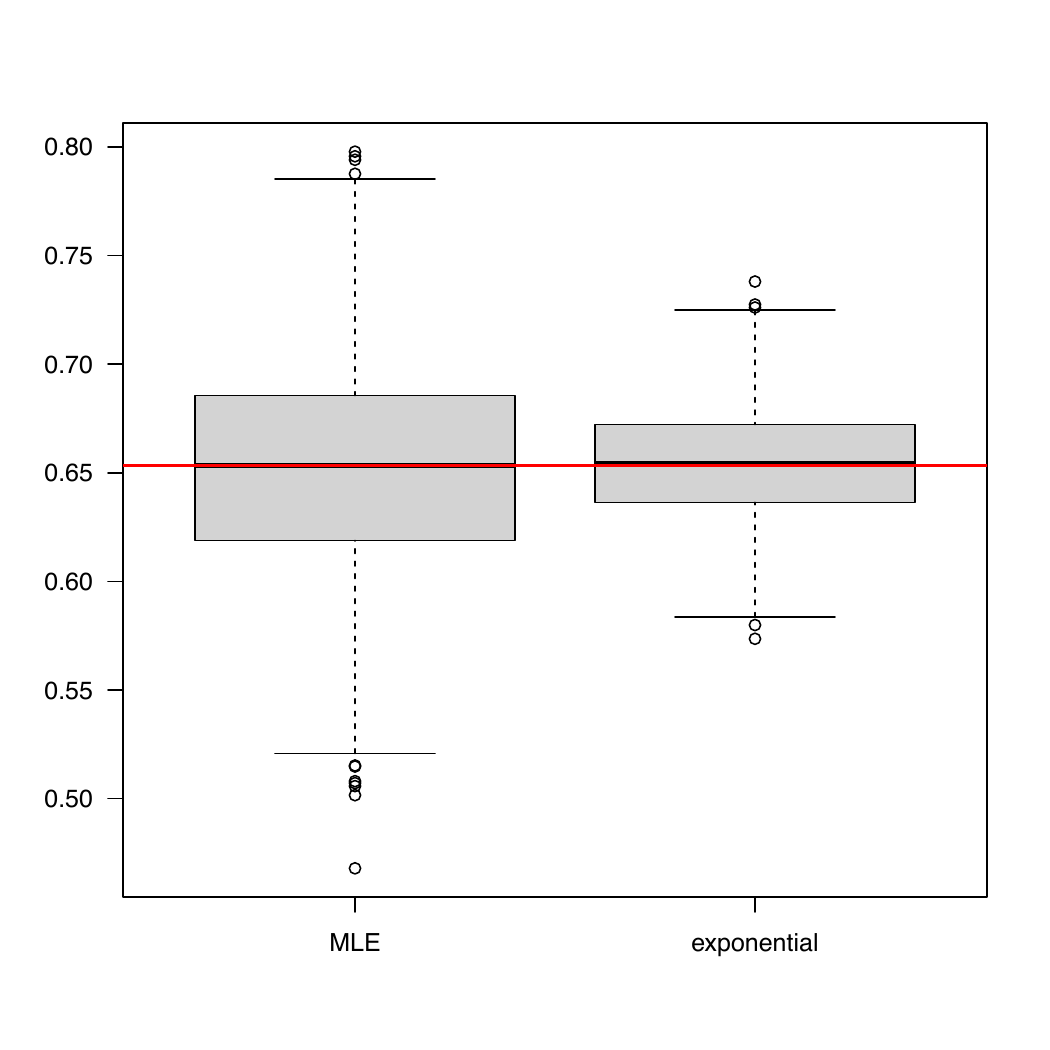}
\caption{}
\label{fig:percentages500_Weibull}
\end{subfigure}
\begin{subfigure}[b]{0.45\textwidth}
\includegraphics[width=\textwidth]{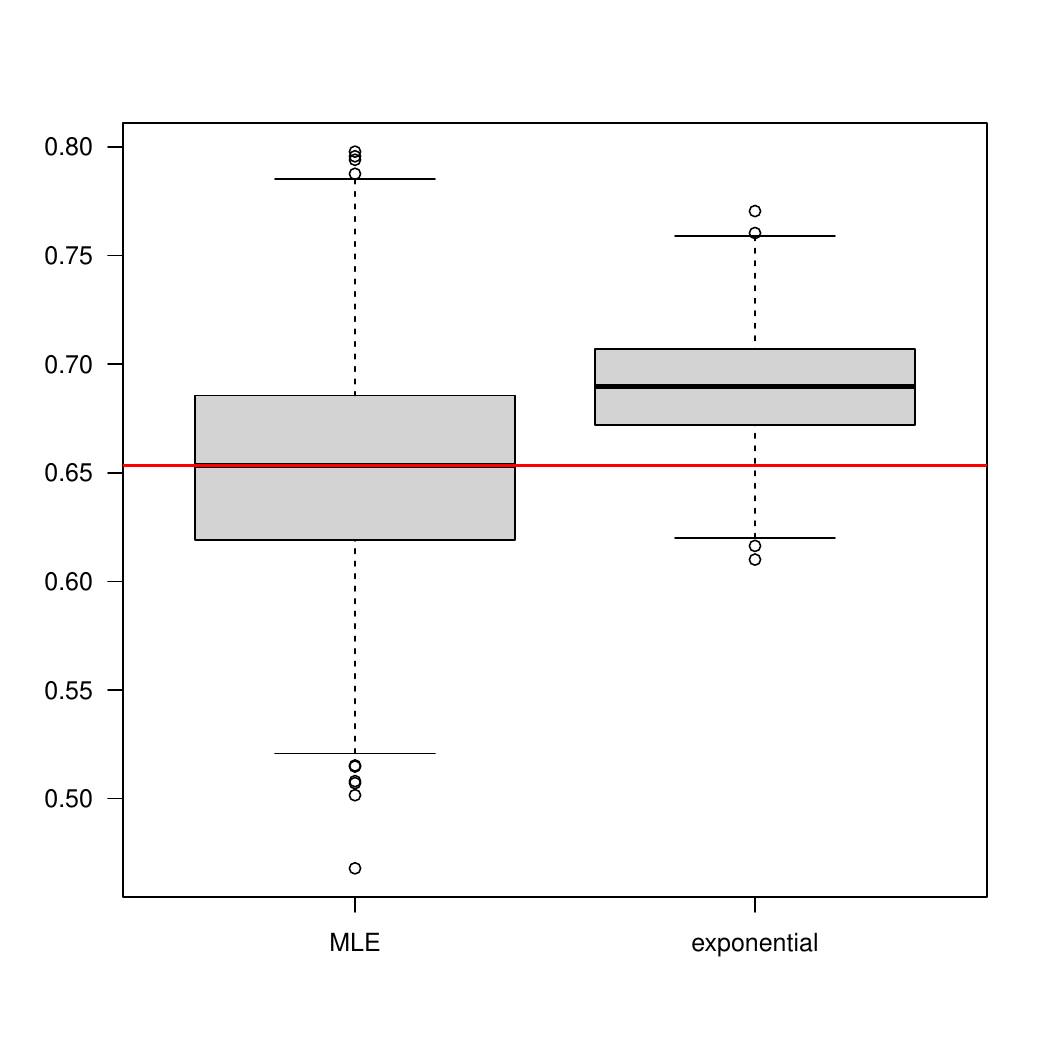}
\caption{}
\label{fig:percentages1000_Weibull}
\end{subfigure}
\caption{(a) Box plot of  $1000$ nonparametric and parametric estimates of $\int_5^{6} F_a(x)\,dx$, where $F_a$ is the truncated exponential distribution function, defined by (\ref{F_a}), with $a=6$, and where the time variables of the sample are rounded to the nearest upper integer. (b) Box plot of  $1000$ nonparametric  estimates of $\int_5^{6} F_a(x)\,dx$ and  $1000$ parametric  estimates of $F_a(6)$, where the time variables are rounded in the same way in the samples.\\
In both cases, the red line segment shows the value of the real $\int_5^6 F_a(x)\,dx$.}
\label{figure:boxplot500}
\end{figure}

The doubly interval censored model, for which the intervals, containing the time of getting symptomatic, are longer than one day, is considered in \cite{lauer:20}.
In this case there is again an interval $[E_L,E_R]$ for the infection time and an interval $[S_L,S_R]$ for the time of becoming symptomatic. One can, just as in \cite{piet:21}, shift the data in such a way that $E_L=0$, which leaves us with three numbers: the time $E$ (``length of Exposure time'') and the times $S_L$ and $S_R$, adapted for the shifting of $E_L$ to zero.
Denoting the (real) time of becoming symptomatic by $S$, we have that $S$ is the sum of the the infection time $I$ and the incubation time $U$. We also assume, conditionally on the exposure time $E$, that $I$ and $U$ are independent and that the time of becoming infected is uniformly distributed on the interval $[0,E]$
Typical schematic pictures of the doubly interval censored model are shown in Figure \ref{fig:picture_double} for two different situations for the interval $[S_L,S_R]$, containing $S$.

\begin{figure}[!ht]
\centering
\includegraphics[width=0.8\textwidth]{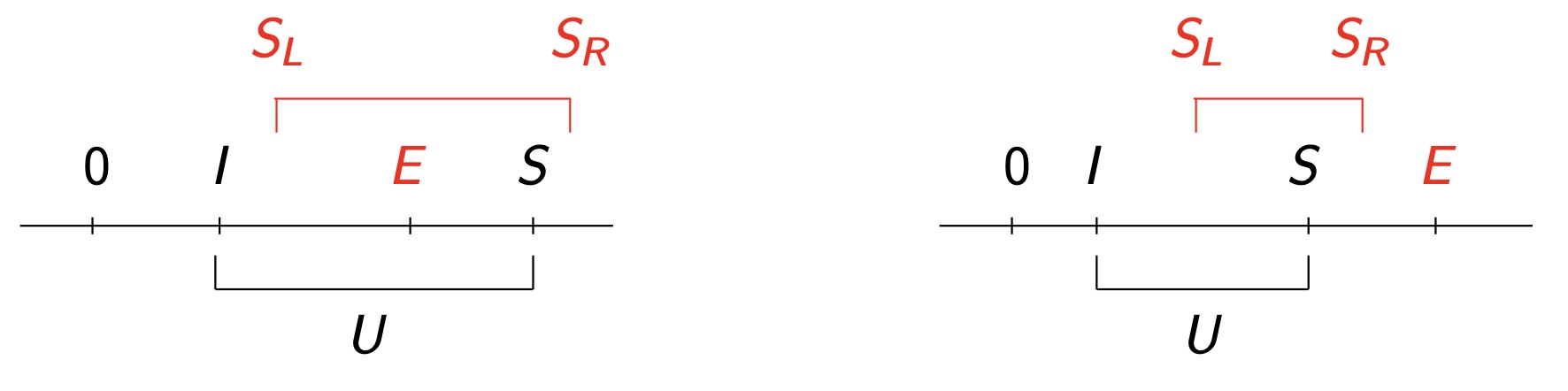}
\caption{Doubly interval censored data. $S$ the time of becoming symptomatic, $I$ infection time and $U$ the (length of the) incubation time. We can only observe $E$ and the interval $[S_L,S_R]$, containing $S$.}
\label{fig:picture_double}
\end{figure}

In this case the log likelihood for one observation is of the form
\begin{align}
\label{log_likelihood_double}\log \left[\F(S_R)-\F(S_L)-\F(S_R-E)+\F(S_L-E)\right],
\end{align}
neglecting parts not affecting the maximization problem, where $\F(u)=\int_0^u F(x)\,dx$, $u>0$.
Such an analysis if given in \cite{lauer:20}. The analysis in \cite{lauer:20} is parametric, considering the log-normal, gamma, Weibull and Erlang distributions as options for the incubation time distribution.

\subsection{Outline of the paper}
\label{subsection:outline}
One can restrict consideration to the estimation of parameters $\bar F_0(i)$ of type (\ref{parameters_discretized}).
We discusss this nonparametric model  in Section \ref{section:discrete}. If one only considers a finite number of parameters of this type, one can specify the asymptotic (normal) distribution, using the Fisher information matrix The rate of convergence of the estimates is $\sqrt{n}$. The results are given in Theorems \ref{th:local_limit_discrete} and  \ref{th:local_limit_discrete_double} of Section \ref{section:discrete}.

Computation of the nonparametric MLE is non-trivial. We discuss this in Section \ref{section:computation}, where the support reduction is introduced for computing the MLE in both models.
In section \ref{section:conf_int} we discuss the construction of confidence intervals for both models. Proofs are given in the appendix.

\section{The nonparametric model}
\label{section:discrete}
We consider the estimation the parameters $\bar F_0(i)$, defined by (\ref{parameters_discretized}). 
The first question is whether the MLE is consistent for these parameters. The answer is affirmative, under some conditions on the distribution function $F_E$ and $F_0$, as the following lemma shows. Here and the rest of the paper, we assume that the $S_i$ and $E_i$ are integers.

\begin{lemma}[Consistency of $\hat F_n$ for $\bar F_0$]
\label{lemma:consistency_single}
Let $\bar F_0$ have strictly positive mass at all points $1,2,\dots,M_1$, where $M_1$ is a integer such that $\bar F_0(i)=1$, for $i>M_1$. Moreover, let $F_E$ have positive mass on all points $1,2,\dots,M_2$, where $M_2>M_1/2$. Then $\hat F_n$ is a consistent estimate of $\bar F_0$.
\end{lemma}

\begin{proof}
The proof parallels the proof for the continuous model in \cite{piet:23}.

The first observation is that $\hat F_n(i)=1$ if $i>M_1$, since then $i>S_j-E_j$ for all pairs $E_j,S_j$ for which $E_j>0$, so the log likelihood becomes largest by putting $F(i)=1$ (there is no compensating $-F(S_j-E_j)$, such that $i\le S_j-E_j$).

Let $\ell(F)$ be defined by
\begin{align*}
\ell(F)=\int \log\{F(s)-F(s-e)\}\,d\Q_n(e,s),
\end{align*}
for distribution functions $F$ on $\R$, which satisfy $F(x)=0$, $x\le0$, where $\Q_n$ is the empirical distribution of the pairs $(E_i,S_i)$. Then we get:
\begin{align*}
\lim_{\e\downarrow0}\e^{-1}\left\{\ell\left((1-\e)\hat F_n+\e F_0\right)-\ell\left(\hat F_n\right)\right\}\le0,
\end{align*}
since $\hat F_n$ is the MLE. The limit exists because of the concavity of $\ell$. Evaluating this limit, we get:
\begin{align*}
\int\frac{F_0(s)-F_0(s-e)}{\hat F_n(s)-\hat F_n(s-e)}\,d\Q_n(e,s)\le1.
\end{align*}

This gives for a limit point $F$ of a subsequence $\hat F_{n_k}$, using Helly's compactness theorem,
\begin{align*}
\int \sum_{i=1}^{M_1+e}e^{-1}\frac{\{\bar F_0(i)-\bar F_0(i-e)\}^2}{F(i)-F(i-e)}\,dF_E(e)\le1,
\end{align*}
where $F_E$ is the (discrete) distribution function of $E$.

We also have:
\begin{align*}
\int \sum_{i=1}^{M_1+e}e^{-1}\{F(i)-F(i-e)\}\,dF_E(e)=1,
\end{align*}
since $F(i)=1$, $i>M_1$. Hence we get:
\begin{align}
\label{consistency_ineq}
\int \sum_{i=1}^{M_1+e}e^{-1}\left\{\frac{\{\bar F_0(i)-\bar F_0(i-e)\}^2}{F(i)-F(i-e)}+\{F(i)-F(i-e)\}\right\}\,dF_E(e)\le2.
\end{align}

Let $j\in\{1,\dots,M_1\}$. Then, if $0<j\le M_1/2$, there is an $e$ with positive $dF_E$-measure, such that $j-e<0$, and then
\begin{align*}
&\frac{\{\bar F_0(j)-\bar F_0(j-e)\}^2}{F(j)-F(j-e)}+F(j)-F(j-e)
=\frac{\{\bar F_0(j)\}^2}{F(j)}+F(j)>2\bar F_0(j),
\end{align*}
unless $F(j)=\bar F_0(j)$.

On the other hand, if $M_1/2<j\le M_1$, there is an $e$ with positive $dF_E$-measure, such that $j+e>M_1$, and then
\begin{align*}
&\frac{\{\bar F_0(j+e)-\bar F_0(j)\}^2}{F(j+e)-F(j)}+F(j+e)-F(j)
=\frac{\{1-\bar F_0(j)\}^2}{1-F(j)}+1-F(j)>2\{1-\bar F_0(j)\},
\end{align*}
unless $F(j)=\bar F_0(j)$.
This means that
\begin{align*}
&\int \sum_{i=1}^{M_1+e}e^{-1}\left\{\frac{\{\bar F_0(i)-\bar F_0(i-e)\}^2}{F(i)-F(i-e)}+\{F(i)-F(i-e)\}\right\}\,dF_E(e)\\
&>2\int \sum_{i=1}^{M_1+e}e^{-1}\{\bar F_0(i)-\bar F_0(i-e)\}\,dF_E(e)=2,
\end{align*}
unless $F=F_0$. The assertion now follows from (\ref{consistency_ineq}).
\end{proof}

We now get the following result where the intervals, containing the time of becoming symptomatic, consist of just one day.
\begin{theorem}
\label{th:local_limit_discrete}
Let $F_E$ have support $\{1,2,\dots,M_2\}$, with positive mass at each $i$ in this set. Let, given $E$, the infection time $I$ have a (continuous) uniform distribution on $[0,E]$, and let the time of getting symptomatic $S$ be the sum of the infection time $I$ and the incubation time $U$, where $I$ and $U$ are independent, given $E$. Let $U$ have an absolutely continuous distribution function $F_0$ on $\R$, such that $F_0(0)=0$ and $F_0(M_1)=1$, where $M_2>M_1/2$, and let $\bar F_0(i)$ be defined by
\begin{align}
\label{barF}
\bar F_0(i)=\int_{i-1}^i F_0(x)\,dx,\qquad i=1,2,\dots.
\end{align}
Moreover, suppose ${\mathcal T}=\{0,1,\dots,m\}$, where $m=M_1+1$, and $p_0(i)\stackrel{\text{\rm def}}=\bar F_0(i)-\bar F_0(i-1)>0$, for each $i=1,\dots,m$, where the $p_0(i)$ satisfy
\begin{align}
\label{sum_p_i}
\sum_{i=1}^m p_0(i)=1.
\end{align}
Furthermore, let ${\mathcal F}$ be the set of right-continuous discrete distribution functions only having jumps at the positive integers
and let $\hat F_n\in{\mathcal F}$ maximize the log likelihood
\begin{align}
\label{lof_likelihood_G}
\ell(F)=\sum_{i=1}^n \log \{F(S_i)-F(S_i-E_i)\},
\end{align}
over $F\in{\mathcal F}$. Finally, let $\hat p_n(i)=\hat F_n(i)-\hat F_n(i-)$ be the corresponding point masses.  Then:
\begin{enumerate}
\item[(i)]
\begin{align}
\label{local_limit_result_discrete}
n^{1/2}\left\{\bigl(\hat p_n(1),\dots,\hat p_n(m-1)\bigr)-\bigl(p_0(1),\dots,p_0(m-1)\bigr)\right\}\stackrel{{\mathcal D}}\longrightarrow N\left(\bm 0,\bm\Sigma^{-1}\right),
\end{align}
where $\bm\Sigma=(\s_{ij})_{i,j,=1,\dots,m-1}$ is the Fisher information matrix with elements
\begin{align}
\label{Fisher_inf_elements}
\s_{ij}=\E\,\frac{\bigl(1_{(S-E,S]}(i)-1_{(S-E,S]}(m)\bigr)\bigl(1_{(S-E,S]}(j)-1_{(S-E,S]}(m)\bigr)}{\{\bar F_0(S)-\bar F_0(S-E)\}^2}\,,
\end{align}
for $i,j=1,\dots,m-1$, and where we assume that $\bm\Sigma$ is nonsingular.
\item[(ii)] Let the covariance matrix $\bm\Sigma$ be defined by (\ref{Fisher_inf_elements}) be nonsingular and let $\hat F_n$ maximize (\ref{lof_likelihood_G}). Then:
\begin{align}
\label{local_limit_result_discrete_df}
n^{1/2}\left\{\bigl(\hat F_n(1),\dots,\hat F_n(m-1)\bigr)-\bigl(\bar F_0(1),\dots,\bar F_0(m-1)\bigr)\right\}\stackrel{{\mathcal D}}\longrightarrow N\left(\bm 0,{\bm A}\bm\Sigma^{-1}{\bm A}^T\right),
\end{align}
where the matrix $\bm A$ has rows $\sum_{j=1}^{i} {\bm e}_j^T$, $i=1,\dots,m-1$, for the unit vectors ${\bm e}_j\in\R^{m-1}$.
\end{enumerate}
\end{theorem}

\begin{remark}
\label{remark:singularity}
{\rm Note that $p_0(m)=1-\sum_{i=1}^{m-1}p_0(i)$ and $\bar F_0(m)=1$ and that we therefore have $m-1$ free parameters, as in the case of the multinomial distribution.
}
\end{remark}

\begin{remark}
\label{remark:finite_support}
{\rm The major difference with the conditions of Theorem 4.1 in \cite{piet:23} is that the maximum likelihood estimators have mass at points of a {\it a fixed finite set}. 
Also note that the estimation of only a fixed finite number of parameters pulls the rate of convergence from cube root $n$ to square root $n$.
}
\end{remark}

We can prove a similar result for the more general doubly interval censored case. This time  the Fisher information matrix consists of the elements
\begin{align*}
\s_{ij}=\E\,\frac{\{\psi(E,S_L,S_R,i)-\psi(E,S_L,S_R,m)\}\{\psi(E,S_L,S_R,j)-\psi(E,S_L,S_R,m)\}}{\left\{\int\psi(E,S_L,S_R,t)\,d\bar F_0(t)\right\}^2},
\end{align*}
for $i,j=1,\dots,m-1$, and where
\begin{align}
\label{psi}
\psi(e,s_L,s_E,t)&=(s_R-t)1_{\{t\in(0,s_R]\}}-(s_L-t)1_{\{t\in(0,s_L]\}}\nonumber\\
&\qquad\qquad-(s_R-e-t)1_{\{t\in(0,s_R-e]\}}+(s_L-e-t)1_{\{t\in(0,s_L-e]\}},
\end{align}

Note that
\begin{align*}
\int_{t\in (s_L,s_R]} \{F(t)-F(t-E)\}\,dt=\int \psi(e,s_L,s_E,t)\,dF(t),
\end{align*}
As noticed in Section \ref{section:computation}, exactly the same support reduction algorithm can be used, with the weights $w_i(j)=\psi(E_i,S_{L,i},S_{R,i},j)$  where $\psi$ is defined by (\ref{psi}), instead of the weights $w_i(j)=1_{(S_i-E_i,S_i]}(j)$ for the singly interval censored model (see (\ref{definition_w})).

It is harder to formulate conditions under which the nonparametric MLE is consistent in this model. In the sinlgly interval censored mdel we can use that we can identify values of $S_i-E_i$ for which $\hat F_n(S_i-E_i)=0$ and values of $S_i$ for which $\hat F_n(S_i)=1$. But if we have and interval $[S_{i,L},S_{i,R}]$ (with length larger than 1), containing the time point for getting symptomatic, this becomes more problematic. We therefore include the condition that $\hat F_n$ is consistent for $\bar F$ in our assumptions. 

We get the following analogue of Theorem \ref{th:local_limit_discrete} for the doubly interval censored model.

\begin{theorem}
\label{th:local_limit_discrete_double}
Let $F_0$ and $F_E$ be distributions with the properties defined in Theorem \ref{th:local_limit_discrete}, and let the time of getting symptomatic $S$ satisfy $S\in[S_L,S_R]$.  Moreover, let $\bar F_0$ and $p_0$   satisfy the same conditions as in Theorem \ref{th:local_limit_discrete}.
Finally, let $\hat F_n$ maximize the log likelihood
\begin{align}
\label{lof_likelihood_G2}
\ell(F)=\sum_{i=1}^n \log \int \psi(E,S_{L,i},S_{R,i},t)\,dF(t),
\end{align}
where $\psi$ is defined by (\ref{psi}), over the set of discrete distribution functions ${\mathcal F}$, defined in Theorem \ref{th:local_limit_discrete}, and let $\hat p_n(i)=\hat F_n(i)-\hat F_n(i-)$ be the corresponding point masses. Suppose $\hat F_n$ is consistent for $\bar F_0$.
Then:
\begin{enumerate}
\item[(i)]
\begin{align}
\label{local_limit_result_discrete_double}
n^{1/2}\left\{\bigl(\hat p_n(1),\dots,\hat p_n(m-1)\bigr)-\bigl(p_0(1),\dots,p_0(m-1)\bigr)\right\}\stackrel{{\mathcal D}}\longrightarrow N\left(\bm 0,\bm\Sigma^{-1}\right),
\end{align}
where $\bm\Sigma=(\s_{ij})_{i,j,=1,\dots,m-1}$ is the Fisher information matrix, assumed non-singular, with elements
\begin{align}
\label{Fisher_inf_elements2}
&\s_{ij}\nonumber\\
&=\E\,\frac{\{\psi(E,S_L,S_E,i)-\psi(E,S_L,S_E,m)\}\{\psi(E,S_L,S_E,j)-\psi(E,S_L,S_E,m)\}}{\left\{\int\psi(E,S_L,S_R,t)\,d\bar F_0(t)\right\}^2}\,,
\end{align}
for $i,j=1,\dots,m-1=M_1$, where $\psi$ is defined by (\ref{psi}).
\item[(ii)] Let the covariance matrix $\bm\Sigma$ be defined by (\ref{Fisher_inf_elements}). Then:
\begin{align}
\label{local_limit_result_discrete_df_double}
n^{1/2}\left\{\bigl(\hat F_n(1),\dots,\hat F_n(m-1)\bigr)-\bigl(\bar F_0(1),\dots,\bar F_0(m-1)\bigr)\right\}\stackrel{{\mathcal D}}\longrightarrow N\left(\bm 0,{\bm A}\bm\Sigma^{-1}{\bm A}^T\right),
\end{align}
where the matrix $\bm A$ has rows $\sum_{j=1}^{i} {\bm e}_j^T$, $i=1,\dots,m-1$, for the unit vectors ${\bm e}_j\in\R^{m-1}$.
\end{enumerate}
\end{theorem}

\section{Computation of the nonparametric maximum likelihood estimators}
\label{section:computation}
In \cite{piet:21} two methods were discussed to compute the nonparametric MLE in the singly interval censored model: the EM algorithm and the iterative convex minorant algorithm. The EM algorithm is excruciatingly slow for this model and for this reason the iterative convex minorant algorithm was used in the simulations. But in the case of doubly interval censored data it is less clear how the iterative convex minorant algorithm should be used, although we could think of ways to apply it in this situation too. However, we will turn to a third method of computing the nonparametric MLE, the {\it support reduction algorithm}, see \cite{grojowe:08}. This method can be applied with equal ease to the two models.

We first discuss the support reduction algorithm for the singly interval censored model. The support reduction algorithm starts by specifying a grid of points $\cS=\{1,\dots,M_1+M_2\}$ which could be points of mass of the MLE. As an example, for the data set analyzed in \cite{piet:21}, one could take set of points $\cS=\{1,2,\dots,30\}$. We can also take the points $S_i$ and $(S_i-E_i)1_{\{S_i-E_i>0\}}$ because these are the points appearing in the log likelihood.

The log likelihood, divided by $n$, for this set of points can be written
\begin{align*}
&\ell(p_1,\dots,p_M)=n^{-1}\sum_{i=1}^n\log\left\{\sum_{j=1}^M p_j1_{\{j\in(S_i-E_i,S_i]\}}\right\},
\end{align*}
where $M=M_1+M_2+1$, $p_j\ge0$, and $\sum_{j=1}^M p_j=1$, and where the subset of locations $j$ of {\it strictly} positive mass $p_j$ have to be estimated. Introducing the notation
\begin{align}
\label{definition_w}
w_i(j)=1_{\{j\in(S_i-E_i,S_i]\}},
\end{align}
we can write the log likelihood, divided by $n$, for $\{p_1,\dots,p_M\}$ at $\{1,\dots,M\}$
\begin{align}
\label{discrete_loglikelihood}
\ell(p_1,\dots,p_M)=n^{-1}\sum_{i=1}^n\log\left\{\sum_{j=1}^{M} p_j w_i(j)\right\}
\end{align}
Turning the maximization problem into a minimization problem on the cone $\R_+^M$, with a Lagrange term to ensure that the solution satisfies $\sum_{i=1}^M p_i=1$, we get as our criterion function
\begin{align}
\label{criterion_function1}
\f(p_1,\dots,p_M)=-n^{-1}\sum_{i=1}^n\log\left\{\sum_{j=1}^{M} p_j w_i(j)\right\}
+\sum_{i=1}^M p_i-1
\end{align}
For this function  have the following lemma.

\begin{lemma}[Fenchel duality conditions for minimization on a cone]
\label{lemma:fenchel}
The function $\f$ in (\ref{criterion_function1}) is minimized on $\R_+^M$ if and only if
\begin{enumerate}
\item[(i)]
\begin{align}
\label{fenchel1}
\frac{\partial}{\partial p_j}\f(p_1,\dots,p_M)\ge0,\qquad j=0,\dots,M,
\end{align}
and
\item[(ii)]
\begin{align}
\label{fenchel2}
\sum_{j=1}^M p_j\frac{\partial}{\partial p_j}\f(p_1,\dots,p_m)=0.
\end{align}
\end{enumerate}
\end{lemma}

The algorithms mentioned (EM, convex minorant algorithm and support reduction algorithm) run until the conditions of Lemma 1 are satisfied up to a certain tolerance, for which we take $10^{-10}$. The EM algorithm tries to do this by simple iteration:
\begin{align*}
p_j'=p_j\left\{1-\frac{\partial}{\partial p_j}\f(p_1,\dots,p_M)\right\},\qquad j=1,\dots,M,
\end{align*}
(see (7) in \cite{piet:21}), the iterative convex minorant algorithm by introducing a quadratic approximation, parametrizing with the values of the distribution function instead of the point masses  (see \cite{piet:21}).

As in the iterative convex minorant algorithm, the support reduction algorithm  employs quadratic approximation, but the parametrization uses the point masses. Expanding the first two terms of the log likelihood at a fixed vector $\bm p^{(0)}=(p_1^{(0)},\dots,p_{M}^{(0)})$, we get:
\begin{align*}
&n^{-1}\sum_{i=1}^n\log\left\{\sum_{j=1}^{M} p_j w_i(j)\right\}
-n^{-1}\sum_{i=1}^n\log\left\{\sum_{j=1}^{M} p_j^{(0)} w_i(j)\right\}\\
&\approx n^{-1}\sum_{i=1}^M\frac{\sum_{j=1}^{M} \left(p_j-p_j^{(0)}\right) w_i(j)}{\sum_{j=1}^{M} p_j^{(0)} w_i(j)}-\tfrac12n^{-1}\sum_{i=1}^n\frac{\left\{\sum_{j=1}^{M} \left(p_j-p_j^{(0)}\right) w_i(j)\right\}^2}{\left\{\sum_{j=1}^{M} p_j^{(0)} w_i(j)\right\}^2}\,.
\end{align*}

We now use iterative minimization.
For fixed $m_0<M$ and a subset $\{j_1,\dots,j_{m_0}\}\subset \{1,\dots,M\}$, we minimize
\begin{align}
\label{criterion_LS}
&\frac1{2n}\sum_{i=1}^n\frac{\sum_{k=1}^{m_0} p_{j_k}^2 w_i(j_k)^2+2\sum_{k<\ell\le m_0}p_{j_k}p_{j_{\ell}}w_i(j_k)w_i(j_{\ell})}
{\left\{\sum_{j=1}^{M} p_j^{(0)} w_i(j)\right\}^2}\nonumber\\
&\qquad\qquad\qquad\qquad\qquad\qquad-\frac2n\sum_{i=1}^n\frac{\sum_{k=1}^{m_0} p_{j_k} w_i(j_k)}{\sum_{j=1}^{M} p_j^{(0)} w_i(j)}+\sum_{k=1}^{m_0} p_{j_k}.
\end{align}
as a function of $p_1,\dots,p_{m_0}$, where, at the start of the iterations
\begin{align*}
m_0=1\qquad\text{ and }\qquad p_j^{(0)}=\frac1M,\qquad j=1,\dots,M,
\end{align*}
and where the double sum in the first numerator vanishes if $m_0=1$.

Then we investigate if adding a point $j_{m_0+1}$ not in the set $\{j_1,\dots,j_{m_0}\}$ and minimizing (\ref{criterion_LS}) over $\{p_{j_1},\dots,p_{j_{m_0+1}}\}$, with $m_0$ replaced by $m_0+1$, leads to a smaller value of (2.1) with the extra point. This may lead to a solution with negative  $p_i$'s. In that case  we remove the point $k$ with the smallest value of $p_k<0$ and solve the least squares minimization problem for (\ref{criterion_LS}) again. It can be proved that this procedure does not remove the point just added again (see, e.g., \cite{mary_meyer:13} and \cite{grojowe:08}). If this solution gives again values $p_i<0$, we reduce the set further to a subset of $m_0-1$ points and solve the least squares minimization problem for (\ref{criterion_LS}) again with $m_0$ replaced by $m_0-1$, continuing until we find a solution with only positive $p_i$'s.

Then we repeat the whole procedure, starting by investigating  whether adding a point $j_{m_0'+1}$ not in the set $j_1,\dots,j_{m_0'}\}$ leads to a smaller value of the criterion for the new subset $\{j_1,\dots,j_{m_0'}\}\subset \{1,\dots,M\}$. Continuing in this way we find a subset $\{j_1,\dots,j_{m_0}\}$ and corresponding $p_{j_1},\dots,p_{j_{m_0}}$ which solves the least squares problem for all possible subsets $\{j_1,\dots,j_{m_0}\}\subset \{1,\dots,M\}$.

Next we change the values $p_j^{(0)}$ in the denominators of (\ref{criterion_LS}). Let
\begin{align*}
\bm p=(p_1,\dots,p_M),\qquad \bm p^{(0)}=\left(p_1^{(0)},\dots,p_M^{(0)}\right),
\end{align*}
where $\bm p$ consists of the values $p_{j_k}$ found in the iterative least squares minimization procedure and zeroes for indices $j_k$ not corresponding to indices of the subset $\{j_1,\dots,j_{m_0}\}$.
Using a line search procedure, for which we use Armijo's rule, we look for a convex combination
\begin{align*}
\bm p'=\a\bm p+(1-\a)\bm p^{(0)},\qquad \a\in(0,1),
\end{align*}
such that $\f(\bm p')<\f(\bm p^{(0)})$, where $\f$ is defined by (\ref{criterion_function1}). Then we set $\bm p^{(0)} := \bm p'$, and repeat the iterative least squares minimization procedure, described above.

We repeat these inner and outer iterations until conditions (i) and (ii) of Lemma \ref{lemma:fenchel} are satisfied up to a certain tolerance, for which we took $10^{-10}$.

As an example, the algorithm is applied to the data set, given in \cite{piet:21}, starting with $p_{i}^{(0)}=1/M$ at the points $\{1,\dots,M\}$, $M=31$, and $p_1=1$ at $t_1=10$.
\begin{align*}
\begin{array}{cccccccccc}
&\text{iteration} &\quad &\text{criterion} &\qquad &\min_{j:p_j>0}\frac{\partial}{\partial p_j}\f(\bm p) &\qquad &\left|\langle \bm p,\f'(\bm p)\rangle\right| &\quad &\#\{j:p_j>0\}\\
&  &    &  &	 &&	&  & &\\
    &1    &   &1.5042265478  & 	&-0.1222076701  &	&0.0650411285   & &7\\
    &2    &   &1.4607858577 &  	&-0.0245250080  &	 &0.0650411285  & &7\\
    &3  &    &1.4528857033  & 	&-0.0016619636  &	 &0.0008604701  & &7\\
    &4   &    &1.4523204985 &  	&-0.0000347676  &	&0.0000174294   & &7\\
    &5   &     &1.4522988585 &  &-0.0000003963  &	&0.0000001969   & &7\\
    &6  &     &1.4522974627 &  	&-0.0000000040  &	&0.0000000020   & &7\\
    &7  &    &1.4522973319  &	 &-0.0000000000 &	&0.0000000000   & &7
    \end{array}
\end{align*}

It is seen that after the first least squares iteration run, the algorithm has found $7$ points of strictly positive mass (the points 3 to 9) and that this number does not change in the following outer iterations. It is also clear that the outer iteration are of (quadratic) Newton type. The end solution coincides in all $10$ decimals with the result of the iterative convex minorant algorithm in \cite{piet:21}. It can be reproduced by running the {\tt R} script for the support reduction algorithm in \cite{github:20}.

The support reduction algorithm can be run in exactly the same way for the {\it doubly interval censored} data with more general intervals. The only change concerns the $w_j(1)$. This time $w_j(i)$ is defined by (\ref{psi}). The log likelihood is again given by (\ref{discrete_loglikelihood}), but with the new definition of the $w_j(i)$. The solution of the maximization problem on the set of points ${\mathcal S}=\{1,\dots,M\}$ is again characterized by Lemma \ref{lemma:fenchel}.

\section{Confidence intervals}
\label{section:conf_int}
To construct confidence intervals for the distribution function of the incubation time $\bar F_0$, discretized on the integers as in (\ref{parameters_discretized}), based on the nonparametric MLE for the singly and doubly interval censoring models, we can use Theorems \ref{th:local_limit_discrete} and \ref{th:local_limit_discrete_double}. Because the Weibull distribution is a popular tool for modeling the incubation time distribution in medical statistics, we use simulations from this distribution as our examples.

Since we have square root $n$ convergence and asymptotic normality, we can also use bootstrap confidence intervals. We do not run into the inconsistency difficulties from which the classical nonparametric bootstrap suffers in the continuous model (see \cite{piet:21} and \cite{piet:23}). Bootstrapping has the advantage that we do not have to estimate the asymptotic variances.

We start by considering asymptotic confidence intervals, using Theorems \ref{th:local_limit_discrete} and \ref{th:local_limit_discrete_double} for estimating the variance. If we want a 95\% confidence intervals for the distribution function at a fixed point $t$, we can use the interval
\begin{align}
\label{conf_int_CLS}
[\hat F_n(t)-1.96\,\hat\s_n(t)/\sqrt{n},\hat F_n(t)+1.96\,\hat\s_n(t)/\sqrt{n}],
\end{align}
where $\hat F_n$ is the nonparametric MLE and $\hat \s_n(t)$ is the square root of a diagonal element of the inverse  {\it observed} Fisher information matrix, corresponding to the Fisher information matrix of Theorems \ref{th:local_limit_discrete} and \ref{th:local_limit_discrete_double}.

The observed Fisher information matrix is defined by $\bm {F}=(f_{jk})$, where
\begin{align}
\label{Fisher_inf_elements_observed}
f_{jk}=n^{-1}\sum_{i=1}^n\frac{\bigl(1_{(S_i-E_i,S_i]}(j)-1_{(S_i-E_i,S_i]}(m)\bigr)\bigl(1_{(S_i-E_i,S_i]}(k)-1_{(S_i-E_i,S_i]}(m)\bigr)}{\{\hat F_n(S_i)-\hat F_n(S_i-E_i)\}^2}\,,
\end{align}
and the $j$ and $k$ are points of mass of the MLE $\hat F_n$. If $1_1,\dots,i_{\ell}$ are the indices of the points of mass, the $(\ell-1)\times (\ell-1)$ matrix ${\bm F}^{-1}$ is the inverse of the corresponding observed Fisher information matrix with elements $f_{jk}$, $j,k=i_1,\dots,i_{\ell-1}$. The matrix $\bm A{\bm F}^{-1}\bm A^T$, where the matrix $\bm A$ has rows $\sum_{j=1}^{i} {\bm e}_j^T$, $i=1,\dots,\ell-1$, for the unit vectors ${\bm e}_j\in\R^{\ell-1}$, is an estimator of the covariance matrix of $(\hat F_n(i_1),\dots,\hat F_n(i_{\ell-1}))$. So the diagonal elements of this matrix are the estimates of the variances of $\hat F_n(i_j)$, $j=1,\dots,\ell-1$.

\begin{figure}[!ht]
\begin{subfigure}[b]{0.45\textwidth}
\includegraphics[width=\textwidth]{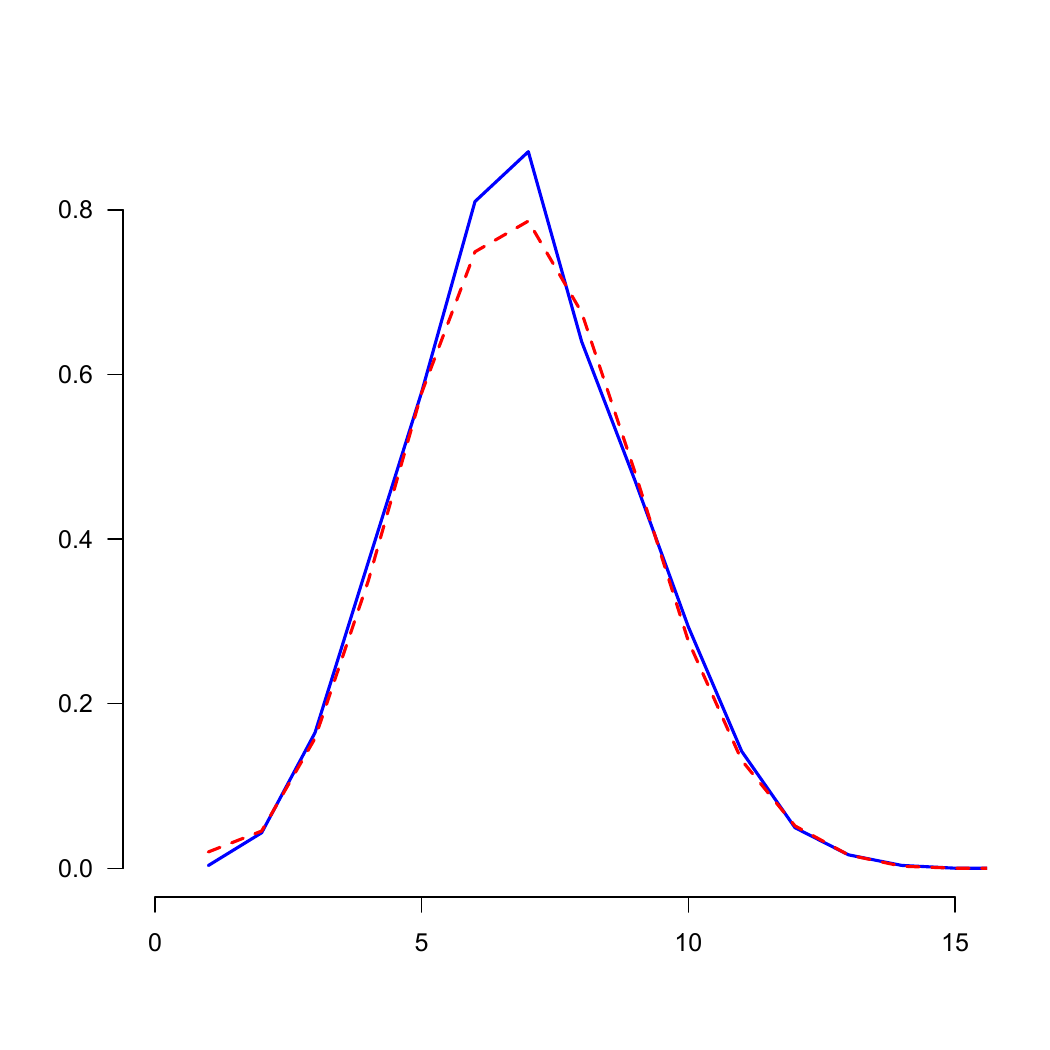}
\caption{}
\label{fig:variances10,000}
\end{subfigure}
\begin{subfigure}[b]{0.45\textwidth}
\includegraphics[width=\textwidth]{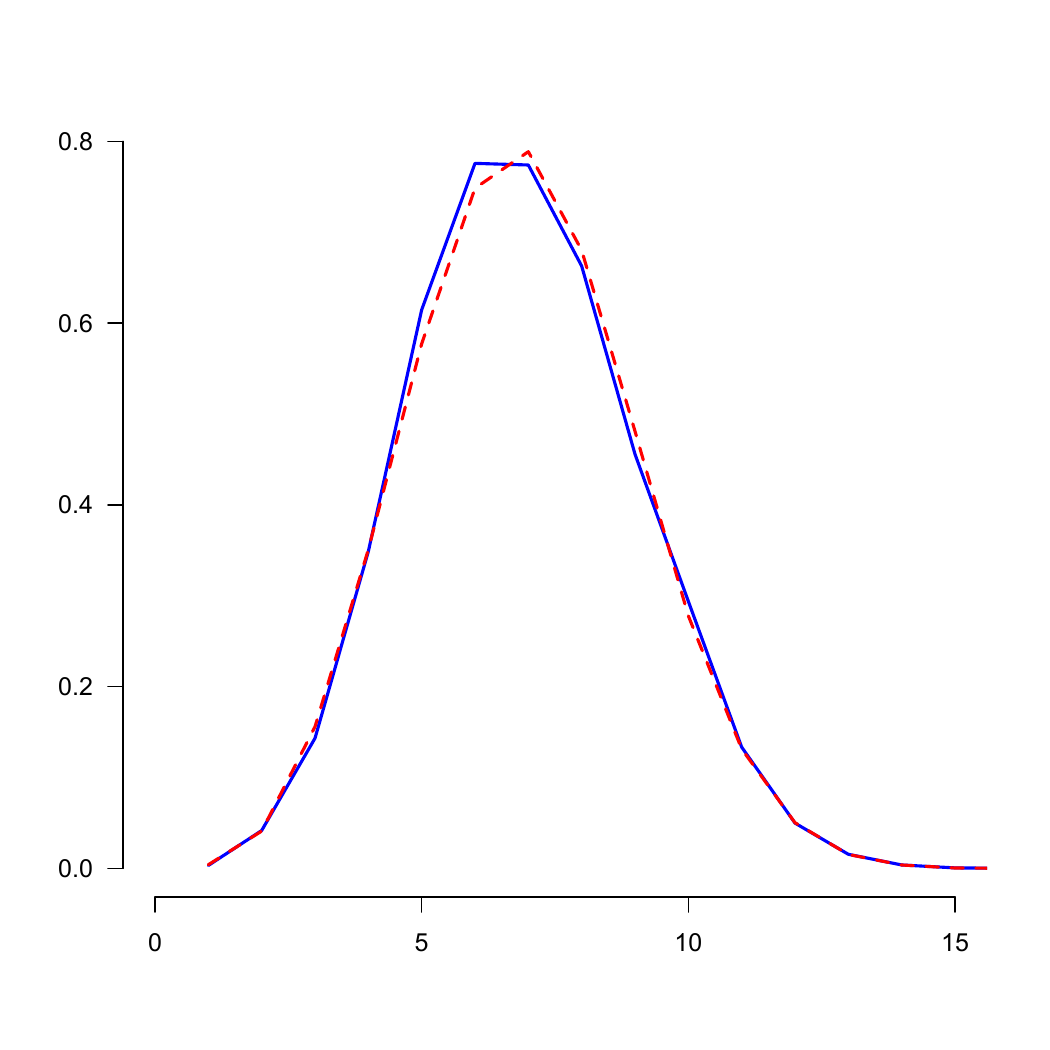}
\caption{}
\label{fig:variances10,000}
\end{subfigure}
\caption{Estimates of the variances of the $\hat F_n(i)$ by $n$ times the actual variances of 1000 samples (blue, solid) and by means of the inverses of the observed Fisher information matrices over the $1000$ samples (dashed, red), (a) for sample size $n=1000$ and (b) for sample size $n=10,000$. We used linear interpolation between the values at the points $1,2,\dots$.}
\label{figure:variances}
\end{figure}

We use these diagonal elements $D_{i_j}$as estimates of all variances, by defining
\begin{align}
\label{extended_diagonal}
D_k=\left\{\begin{array}{lll}
0,\qquad &,k<i_1,\\
D_{i_j} &,i_j\le k<i_{j+1}\,\qquad ,1\le j <\ell,\\
0, &,k\ge i_{\ell}.
\end{array}
\right.
\end{align}

The diagonal elements $D_i$ in the extended definition (\ref{extended_diagonal}), are used as estimates of the variances of $\hat F_n(1),\dots,\hat F_n(M_1)$.
The fit of the estimates with $n$ times the actual variances of the $\hat F_n(i)$ over $1000$ samples is remarkably good for our simulation model, see Figure \ref{figure:variances}.

The resulting confidence intervals for a sample of size $n=1000$ is shown in Figure \ref{figure:conf_int_CLT1000} at the points $3$ to $10$, where the MLE puts most of its mass. The coverage of these intervals is also shown. Here we generated $1000$ samples of size  $n=1000$, and computed the fraction of times the parameters $\bar F_0(i)=\int_{i-1}^{i}F_0(t)\,dt$ were inside the intervals (\ref{conf_int_CLS}) at the points 3 to 10.

\begin{figure}[!ht]
\begin{subfigure}[b]{0.45\textwidth}
\includegraphics[width=\textwidth]{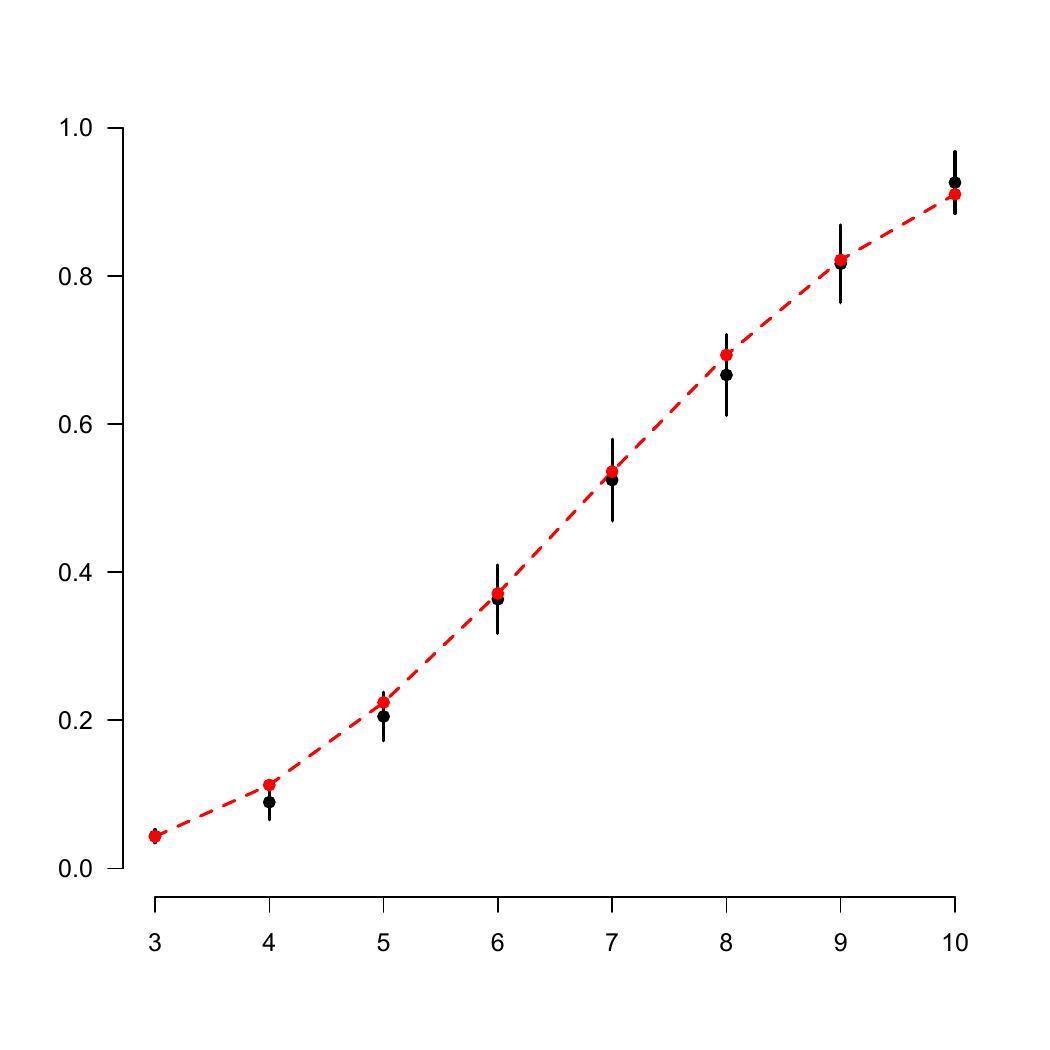}
\caption{}
\label{fig:percentages500_Weibull}
\end{subfigure}
\begin{subfigure}[b]{0.45\textwidth}
\includegraphics[width=\textwidth]{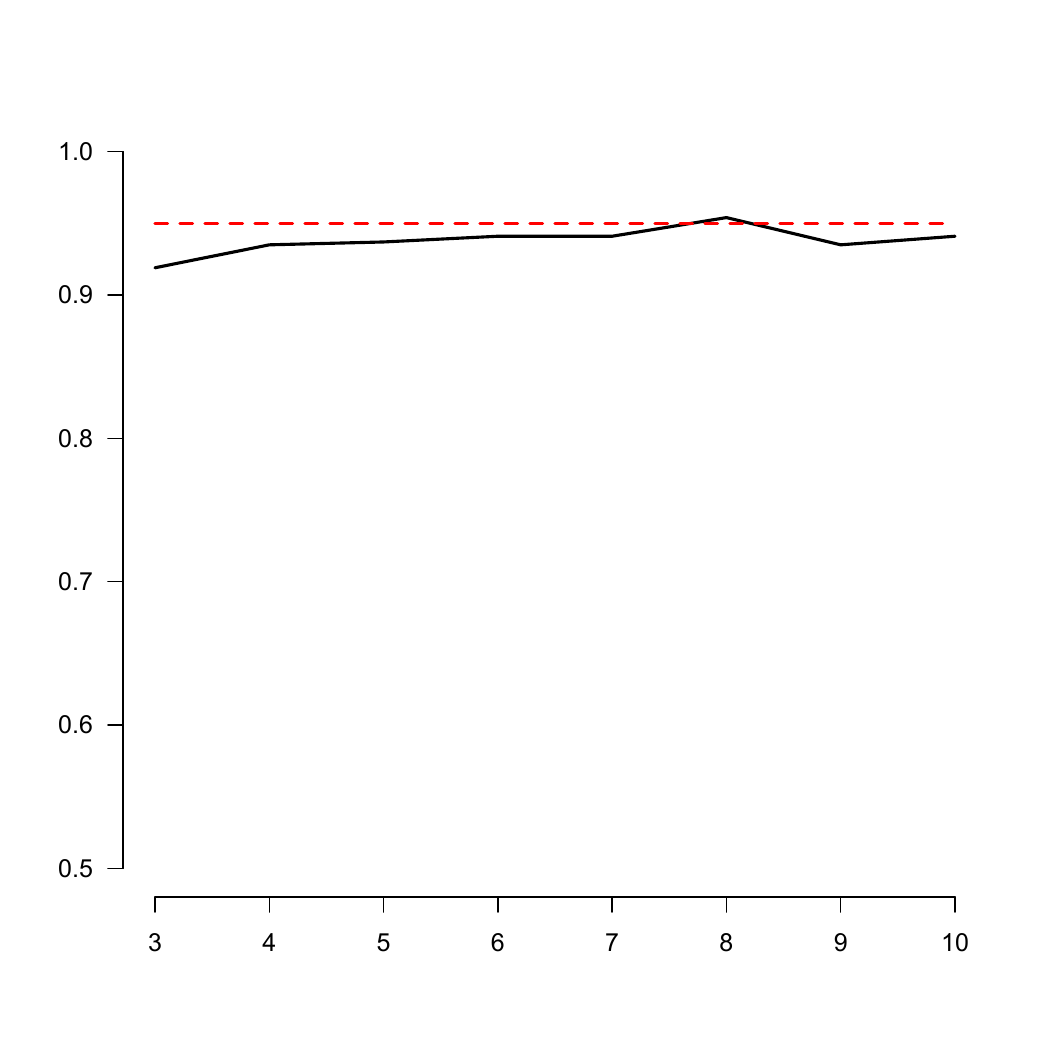}
\caption{}
\label{fig:percentages1000_Weibull}
\end{subfigure}
\caption{(a) 95\% confidence intervals in the singly interval censored model, using (\ref{conf_int_CLS}), for the values of $\bar F_0(i)=\int_{i-1}^{i} F_0(x)\,dx$ (red dots and linearly interpolated dashed red curve) at the points $3,4,\dots,10$ for a sample of size $n=1000$, where $F_0$ is the Weibull distribution function, with parameters $a=3.035$ and $b=0.0026$, truncated at $M_1=15$. The black dots are the values of $\hat F_n$ at these points. (b) Coverage percentages of the  95\% confidence intervals at the points $3,4,\dots,10$, using (\ref{conf_int_CLS}), for sample size $n=1000$.}
\label{figure:conf_int_CLT1000}
\end{figure}

We can run a bootstrap experiment to generate confidence intervals of this type in the following way. We resample with replacement from the data $(E_i,S_i)$ $1000$ samples of the same size $n$ and compute for each of these bootstrap samples the MLE. This gives $1000$ bootstrap values $\hat F_n^*(t)-\hat F_n(t)$. For these bootstrap values of $\hat F_n^*(t)-\hat F_n(t)$ we compute the $0.025$ and $0.975$ quantiles $Q_{0.025}^*(t)$ and $Q_{0.975}^*(t)$, respectively. This gives the bootstrap 95\% confidence intervals
\begin{align}
\label{conf_int_bootstrap}
[\hat F_n(t)-Q_{0.975}^*(t),\hat F_n(t)-Q_{0.025}^*(t)].
\end{align}
The results are shown in Figure \ref{figure:conf_int_single_bootstrap}. It is seen that the results are similar to the results of the method, using the inverse observed Fisher information matrix for generating the confidence intervals.

\begin{figure}[!ht]
\begin{subfigure}[b]{0.45\textwidth}
\includegraphics[width=\textwidth]{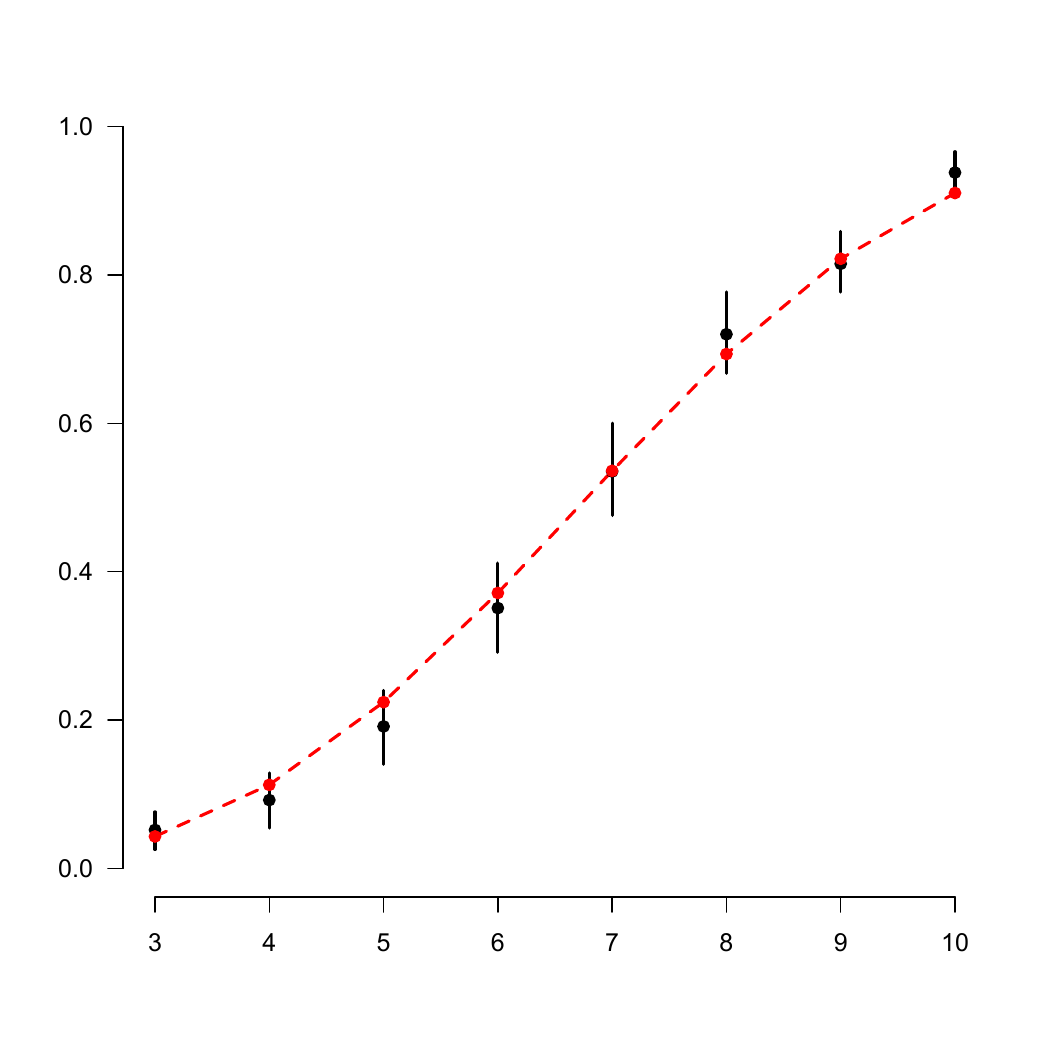}
\caption{}
\label{fig:CI_single1000_bootstrap}
\end{subfigure}
\begin{subfigure}[b]{0.45\textwidth}
\includegraphics[width=\textwidth]{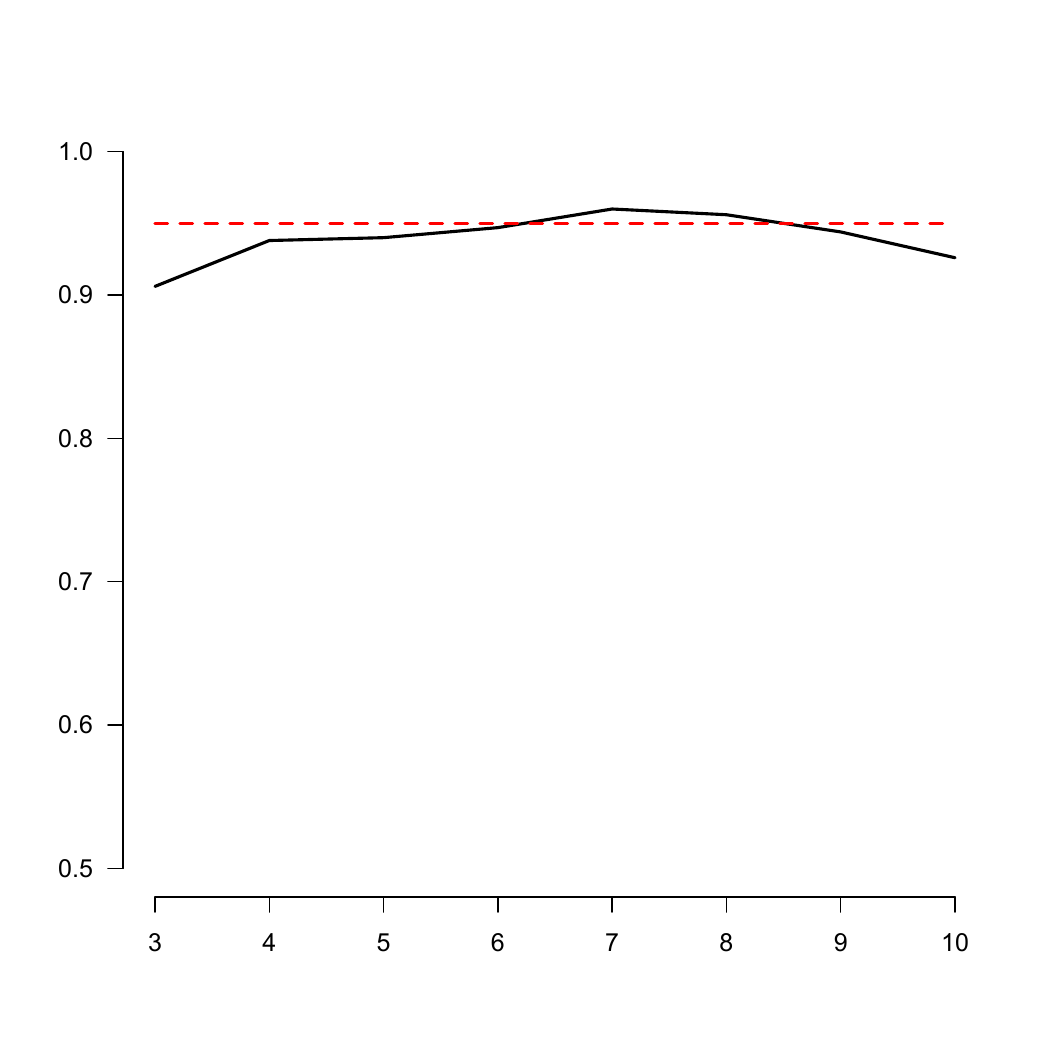}
\caption{}
\label{fig:percentages_single1000_bootstrap}
\end{subfigure}
\caption{(a) 95\% bootstrap confidence intervals in the singly interval censored model, using (\ref{conf_int_bootstrap}), for the values of $\bar F_0(i)=\int_{i-1}^{i}F_0(x)\,dx$  at the points $3,4,\dots,10$ (red dots and linearly interpolated dashed red curve) for a sample of size $n=1000$, where $F_0$ is the Weibull distribution function, truncated at $M_1=15$. The black dots are the values of $\hat F_n$ at these points. (b) Coverage percentages of the bootstrap 95\% confidence intervals at the points $3,4,\dots,10$, using (\ref{conf_int_bootstrap}), for sample size $n=1000$.}
\label{figure:conf_int_single_bootstrap}
\end{figure}

We also simulated data for the doubly censored model. Here we took $S_R$ discretely uniform on $\{\lceil S\rceil,\dots,\lceil S\rceil+3\}$ and $S_L$ discretely uniform on $\{\lfloor S\rfloor,\dots,\lfloor S\rfloor-3\}$ (replacing  $\lfloor S\rfloor-i$  by $0$ if $\lfloor S\rfloor-i<0$). This time the observed Fisher information matrix is defined by
\begin{align}
\label{Fisher_inf_elements_double}
f_{jk}=n^{-1}\sum_{i=1}^n\frac{\tilde \psi(E_i,S_{L,i},S_{R,i},j)\tilde \psi(E_i,S_{L,i},S_{R,i},k)}{\left\{\int\psi(E_i,S_{L,i},S_{R,i},t)\,d\hat F_n(t)\right\}^2}\,,
\end{align}
where
\begin{align*}
\tilde \psi(E_i,S_{L,i},S_{R,i},t)=\psi(E_i,S_{L,i},S_{R,i},t)-\psi(E_i,S_{L,i},S_{R,i},m),
\end{align*}
and where the $j$ and $k$ are points of mass of the MLE $\hat F_n$ and $\psi$ is defined by (\ref{psi}).  The diagonal elements $D_{i_j}$ of the matrix ${\bm A}{\bm F}^{-1}{\bm A}^T$ are extended as in (\ref{extended_diagonal}) and used as estimates of the  variances of the $\hat F_n(i)$.
 
\begin{figure}[!ht]
\begin{subfigure}[b]{0.45\textwidth}
\includegraphics[width=\textwidth]{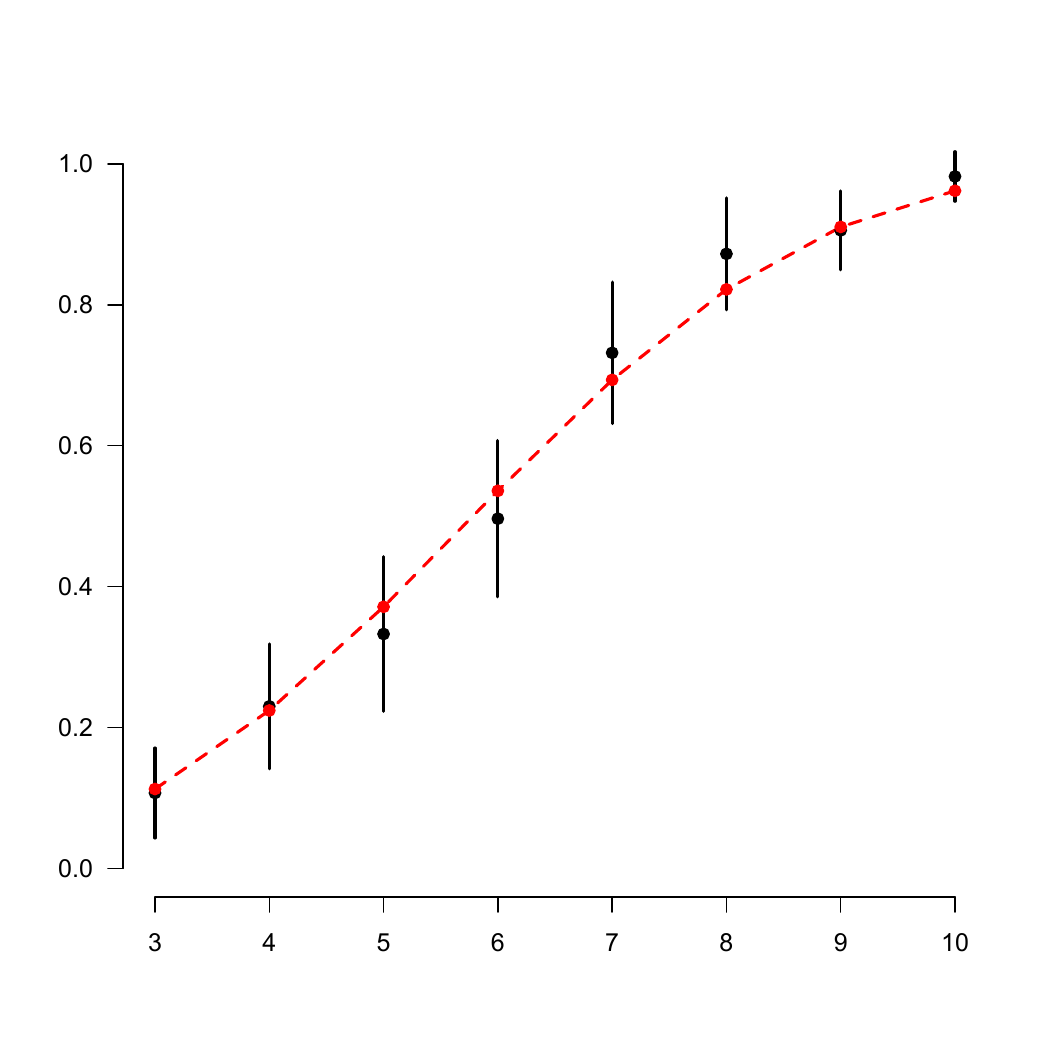}
\caption{}
\label{fig:Fisher_observed500}
\end{subfigure}
\begin{subfigure}[b]{0.45\textwidth}
\includegraphics[width=\textwidth]{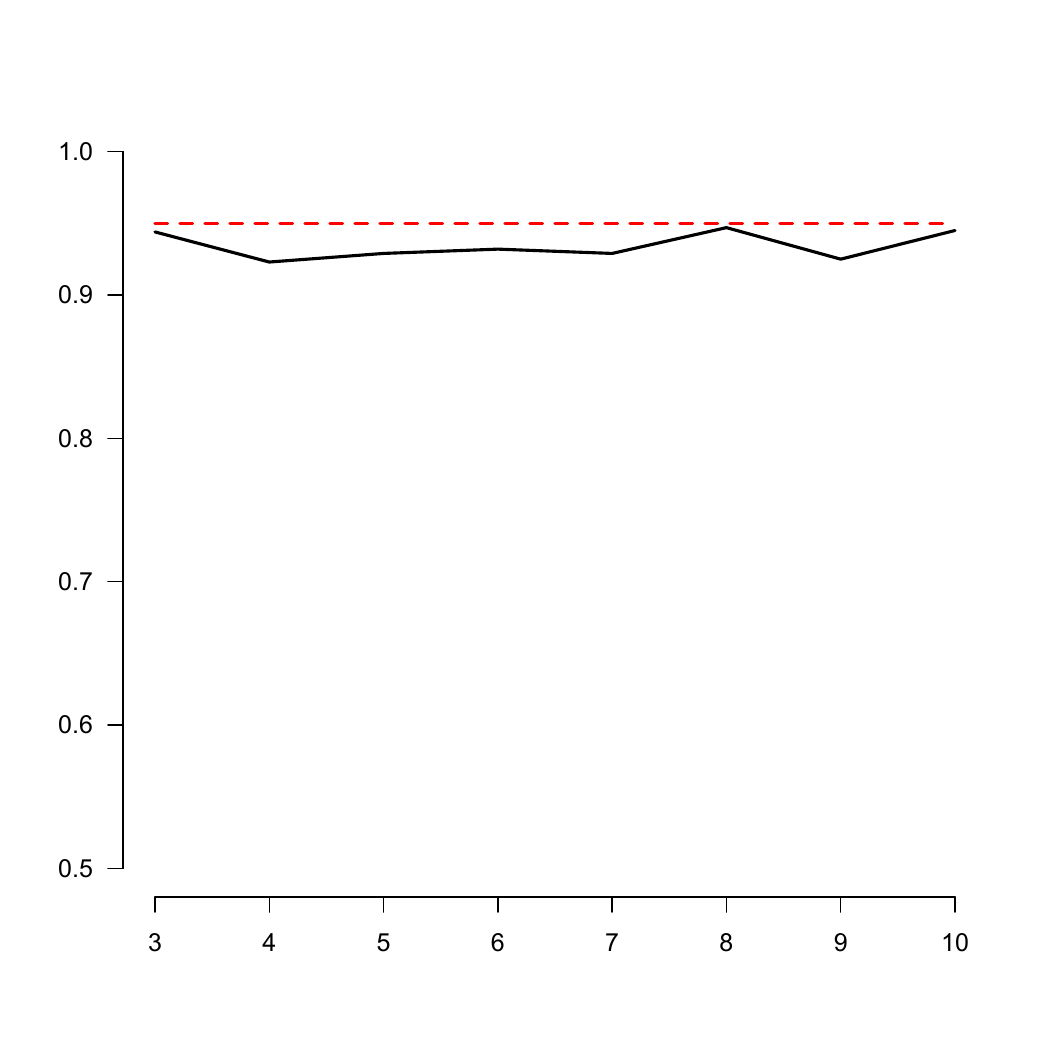}
\caption{}
\label{fig:Fisher_mean500_bootstrap_example}
\end{subfigure}
\caption{95\% confidence intervals for a doubly interval censored simulation sample at the points $3,\dots,10$, using (\ref{conf_int_CLS}), with the variances estimated by the means of the diagonal of the inverse observed Fisher information matrices over $1000$ samples. (b) Coverage percentages for the same example over $1000$ samples, using the same estimates of the variances. Sample size is $n=1000$.}
\label{figure:variances_Fisher_observed}
\end{figure}

The coverages, based on the Fisher information matrix from one sample were not very good this time, but were rather satisfactory if we estimate the Fisher information matrix by the mean of these matrices over $1000$ samples. See Figure \ref{figure:variances_Fisher_observed}. In practice we can take the mean over $1000$ bootstrap samples.

\section{Conclusion}
\label{section:conclusion}
We studied the nonparametric maximum likelihood estimator of the incubation time, which is such an important parameter in the Covid-19 pandemic. The incubation time in our model is part of the sum of the infection time $I$ and the incubation time $U$. Usually the only really (partly) observable quantity is the time of becoming symptomatic, which is given by $S=I+U$. So our variable of interest $U$ has to be pulled out of this sum by deconvolution. If $S$ is observable, one speaks of the {\it singly interval censored model} (the beginning of the incubation time lies in an interval which represents the exposure time and this is the singly interval censored part). It seems most reasonable to assume that the time till infection and the incubation time have continuous distributions, and one can analyze this model under the assumption that $S$ is exactly observable. In that case the MLE of the distribution function of the incubation time converges, under some conditions, at cube root $n$ rate to Chernoff's distribution, see \cite{piet:23}.

However, most of the time $S$ is not directly observable, but only an interval $[S_L,S_R]$ is available, which we know to contain $S$. Taking into account that the observations are usually rounded to days, one can restrict oneself to estimating the means over one day, represented by
\begin{align}
\label{integral_par}
\bar F_0(i)=\int_{i-1}^{i} F_0(t)\,dt,
\end{align}
if $F_0$ is the distribution function of the incubation time. Since this gives us a fixed bounded number of parameters, we can use classical theory, connecting maximum likelihood with the Fisher information, to derive asymptotic distribution theory for the maximum likelihood estimators of these parameters.

We applied the theory, developed in Section \ref{section:discrete} to construct confidence intervals, either by using Theorems \ref{th:local_limit_discrete} and \ref{th:local_limit_discrete_double} directly, or by using a bootstrap method. In a sense, this purely nonparametric method lies at the other extreme of the parametric methods. If one wants to estimate the density, one will have to use some kind of smoothing, as was done in \cite{piet:21}, which is an intermediate method that is still nonparametric and avoids the need to choose between several parametric models. 

The support reduction algorithm seems at present to be the most stable method to estimate the parameters. The computing of the MLE and the confidence intervals is implemented in \cite{github:20} and discussed in Section \ref{section:computation}. {\tt R} scripts for running the algorithms discussed in this paper are available at \cite{github:20}.

\section{Appendix. Proofs}
\label{section:appendix}
In contrast with the difficulties of the continuous model in \cite{piet:23}, the estimation of the parameters $\int_{i-1}^{i} F_0(t)\,dt$ seems rather standard, once we have figured out the relation between the continuous model and the discrete observations of the times of becoming symptomatic. We still have to deal with a deconvolution problem, which we do by using nonparametric maximum likelihood.

\begin{proof}[Proof of Lemma \ref{lemma:fenchel}]
\label{remark:lemma_fenchel}
The key step is to reduce the maximization on the simplex $\{\bm p=(p_1,\dots,p_m)\in\R_+^m:\sum_{i=1}^m p_i=1\}$ to maximization on the cone $\R_+^m$. One can check that minimizing minus the log likelihood (\ref{discrete_loglikelihood}) under the restriction $\sum_{i=1}^mp_i=1$ is equivalent to minimizing (\ref{criterion_function1}), which is the criterion function $+$ a Lagrange term with Lagrange multiplier $\l=1$, on $\R^m_+$. Then the necessary and sufficient conditions for the minimum follow from Fenchel's duality theorem, see \cite{rockafellar:70}, Theorem 31.4.
\end{proof}

\begin{proof}[Proof of Theorem \ref{th:local_limit_discrete}] The log likelihood is of the form
\begin{align*}
\ell(p_1,\dots,p_m)=\sum_{i=1}^n\log\sum_{j=1}^m p_j 1_{((S_i-E_i)_+,S_i]}(j),
\end{align*}
so we count the number of times the point of mass $j$ belongs to an interval $(S_i-E_i)_+,S_i]$, where $S_i$ and $E_i$ are integers, and multiply this with the probability $p_j$. We have
\begin{align}
\label{score_function}
&\frac{\partial}{\partial p_j}\ell\biggl(p_1,\dots,1-\sum_{k=1}^{m-1}p_k\biggr)\nonumber\\
&=\sum_{i=1}^n\frac{1_{((S_i-E_i)_+,S_i]}(j)-1_{((S_i-E_i)_+,S_i]}(m)}{\sum_{k=1}^m p_j 1_{((S_i-E_i)_+,S_i]}(k)}\,,\qquad j=1,\dots,m-1,
\end{align}
and
\begin{align*}
&\frac{\partial^2}{\partial p_j\partial p_{l}}\ell\biggl(p_1,\dots,1-\sum_{k=1}^{m-1}p_k\biggr)\\
&=-\sum_{i=1}^n\frac{\left\{1_{((S_i-E_i)_+,S_i]}(j)-1_{((S_i-E_i)_+,S_i]}(m)\right\}\left\{1_{((S_i-E_i)_+,S_i]}(l)-1_{((S_i-E_i)_+,S_i]}(m)\right\}}{\left\{\sum_{k=1}^m p_j 1_{((S_i-E_i)_+,S_i]}(k)\right\}^2},
\end{align*}
for $j,l=1,\dots,m-1$, using the convention $0/0=0$.

By the assumptions on $F_0$ and $F_E$, the variables $S_i$ and $(S_i-E_i)_+$ will be such that for large $n$, the score functions (\ref{score_function}) will be zero for $p_j=\hat p_j$, where $\hat{\bm p}=(\hat p_1,\dots,\hat p_m)$ is the MLE of $\bm p_0=(p_0(1),\dots,p_0(m))$ (that is: no isotonization is needed), and the result now follows from standard theory.
\end{proof}

\begin{proof}[Proof of Theorem \ref{th:local_limit_discrete_double}]
The proof is entirely similar to the proof of Theorem \ref{th:local_limit_discrete}, but this time the score functions are given by
\begin{align}
\label{score_function_double}
&\frac{\partial}{\partial p_j}\ell\biggl(p_1,\dots,1-\sum_{k=1}^{m-1}p_k\biggr)\nonumber\\
&=\sum_{i=1}^n\frac{\psi(E_i,S_{L,i},S_{R,i},j)-\psi(E_i,S_{L,i},S_{R,i},m)}{\sum_{k=1}^m\psi(E_i,S_{L,i},S_{R,i},t)\,p_k}\,,\qquad j=1,\dots,m-1,
\end{align}
where $\psi$ is defined by (\ref{psi}).

A key part of the treatment of the doubly censored case is the rewrite of the log likelihood for one observation, using the integration by parts
\begin{align*}
\int_{t\in (s_L,s_R]} \{F(t)-F(t-E)\}\,dt=\int \psi(e,s_L,s_E,t)\,dF(t),
\end{align*}
where $\psi$ is defined by (\ref{psi}).
\end{proof}

\bibliographystyle{imsart-nameyear}
\bibliography{cupbook}

\end{document}